\documentclass{amsart}
 \usepackage[foot]{amsaddr}

\usepackage{amsmath}
\usepackage{amssymb}
\usepackage{mathtools}
\usepackage{cite}
\usepackage[alphabetic]{amsrefs}
\usepackage{stmaryrd}
\usepackage[mathscr]{euscript}
\usepackage{amsthm}
\usepackage{enumerate}
\usepackage{hyperref}
\usepackage{caption}

\newtheorem{thm}{Theorem}[section]
\newtheorem{cor}{Corollary}[section]
\newtheorem{lem}{Lemma}[section]
\newtheorem{prop}{Proposition}[section]
\newtheorem{clm}{Claim}[section]
\newtheorem{defn}{Definition}[section]
\newtheorem{rem}{Remark}[section]

\newcommand{\comment}[1]{}

\newcommand{\B}{\mathbb{B}}
\newcommand{\R}{\mathbb{R}}

\newcommand{\C}{\mathbb{C}}

\newcommand{\clo}{\overline}
\newcommand{\on}{\colon}

\newcommand{\inv}[1]{{#1}^{-1}}

\newcommand*{\defeq}{\mathrel{\vcenter{\baselineskip0.5ex \lineskiplimit0pt
                     \hbox{\scriptsize.}\hbox{\scriptsize.}}}%
                     =}

\makeatletter
\@namedef{subjclassname@2020}{\textup{2020} Mathematics Subject Classification}
\makeatother               

\title{Stability of Convex Disks}
\author{Hunter Stufflebeam}
\address{The University of Pennsylvania, Department of Mathematics, David Rittenhouse Lab., 209 South 33rd Street, Philadelphia, PA 19104, USA.}
\email{hstuff(at)sas(dot)upenn(dot)edu}

\begin{document}

\begin{abstract}
We prove that topological disks with positive curvature and strictly convex boundary of large length are close to round spherical caps of constant boundary curvature in the Gromov-Hausdorff sense. This proves stability for a theorem of F.\ Hang and X.\ Wang in \cite{HW}, and can be viewed as an affirmative answer to a convex stability version of the Min-Oo Conjecture in dimension two. As an intermediate step, we obtain a compactness result for a Liouville-type PDE problem.
\end{abstract}

\maketitle


\section{Introduction}

Inequalities in geometric analysis, such as the isoperimetric and systolic, Faber-Krahn and Penrose, relate given geometric objects to understood model cases, taking as input data bounds on curvatures, volumes, eigenvalues, etc. Via such relationships, much work has been done to understand the structure of spaces with natural geometric conditions phrased in terms of such quantities. 

Given an inequality for which one has some understanding of extremizers (the geometric objects which realize equality), one might ask if an object \emph{nearly} realizing equality must somehow share characteristics with the extremizer(s). The first problem of understanding the extremizers might be called a \emph{rigidity problem}. The second problem of understanding near extremizers might be called a \emph{stability problem}. A classical example is the isoperimetric inequality in Euclidean space. Its extremizers are balls, and the stability problem has received much attention--for example consider the work of Fusco-Maggi-Pratelli in \cite{FMP}. 

The aim of this paper is to establish the following stability theorem for convex disks. Throughout, given a 2-manifold $(M,g)$ we use the notation $K_g$ for the Gaussian curvature, $\kappa_g$ for the geodesic curvature of the boundary, and $L_{g}$ for the length functional of $g$. The Gromov-Hausdorff metric on the space of compact metric spaces is denoted by $d_{GH}$.

\begin{thm}[Stability of the Convex Spherical Cap]\label{main}
Fix $c>0$ and let $\delta>0$. Then there exists an $\varepsilon=\varepsilon(c,\delta)>0$ such that if $(M, g)$ is a two dimensional manifold with $K_{g}\geqslant 1$, $\kappa_{g}\geqslant c>0$, and $L_{g}(\partial M)\geqslant 2\pi(1-\varepsilon)/\sqrt{1+c^2}$, then \[d_{GH}((M,g),\clo{\mathbf{B}}_{\inv\cot(c)})<\delta,\] where $\clo{\mathbf{B}}_{\inv\cot(c)}$ is a closed geodesic disk of radius $\inv\cot(c)$ in the round sphere $\mathbb S^2$.
\end{thm}

This result was partially motivated by the desire for a stability theorem corresponding to an old result of V.\ Toponogov:

\begin{thm}[V.\ Toponogov]\label{top}
Let $(M, g)$ be a closed surface with $K_g\geqslant 1$. Then any simple closed geodesic has length bounded from above by $2\pi$, and this length is attained if and only if $(M,g)$ is isometric to the round sphere. 
\end{thm}

There are at least two proofs\footnote{In fact, there are at least three--the paper \cite{newtop} was recently brought to the author's attention.} of this result--the original due to Toponogov using the celebrated triangle comparison theorem, and a modern PDE proof due to F.\ Hang and X.\ Wang (see \cite{K} and \cite{HW}, respectively). In both arguments, one cuts open the closed manifold along the largest closed geodesic to obtain two disks with geodesic boundary. The length inequality, and a corresponding rigidity theorem, is then proven for these disks. The PDE proof of the result applies immediately to the case of convex disks in general:

\begin{thm}[F.\ Hang and X.\ Wang \cite{HW}]\label{HW}
Let $(M, g)$ be a smooth, compact surface with boundary. Suppose $K_g\geqslant 1$ and $\kappa_g\geqslant c\geqslant 0$. Then $L_g(\partial M)\leqslant 2\pi/\sqrt{1+c^2}$. Moreover, equality holds if and only if $(M,g)$ is isometric to $\clo{\mathbf{B}}_{\inv\cot(c)}$. 
\end{thm}

To attempt a proof of stability for Toponogov's theorem, it is thus natural to try proving stability in the disk rigidity theorem of F.\ Hang-X.\ Wang. However, while stability in the convex case does hold as is shown by our following work, this does not extend to the case of disks with weakly convex boundary. And indeed, stability does not generally hold in Toponogov's theorem:  

\begin{rem}[Failure of Stability in Toponogov's Theorem]
Fix any small $\varepsilon>0$ and $\delta>0$. There exists a rotationally symmetric metric $g=\mathrm{d}r^2+\phi^2(r)\mathrm{d}\theta^2$ on $\mathbb S^2$ which has $K_g\geqslant 1$, a simple closed geodesic of length $2\pi-\varepsilon$, and \[d_{GH}((\mathbb S^2, g),(\mathbb S^2, g_{rd}))\geqslant \delta.\] Here, $g_{rd}$ denotes the round metric on $\mathbb S^2$.
\end{rem}

This follows an idea of \cite{CC} \cite{CM}, and was explored in some detail in \cite{WZ}. One takes a good football metric on $\mathbb S^2/\mathbb Z_k$ and carefully smoothes the tips by gluing in spherical caps and rescaling. The interested reader is encouraged to look to the latter source for the specifics of the construction, and it is not hard to deduce from it the claimed properties in the remark above. 

We lastly remark that, in fact, F.\ Hang and X.\ Wang's result in the geodesic boundary case (which is contained in V.\ Toponogov's original proof of Theorem \ref{top}) is an affirmative answer to the famous Min-Oo Conjecture in dimension two\footnote{See \cite{brendle} for a good overview of this story.}, which is an analogue of the Positive Mass Theorem of R.\ Schoen and S.\ T.\ Yau \cite{SY} in spherical geometry. While the Conjecture is known to hold in higher dimensions in many special geometries (see \cite{brendle} again), S.\ Brendle, F.\ C.\ Marques, and A.\ Neves showed in \cite{BMN} that for every dimension at least three, counterexamples to the Conjecture as stated exist. Insofar as the stability question for various incarnations of the Positive Mass Theorem is an important driving industry in modern geometric analysis, it is natural to wonder whether or not the one true case of Min-Oo's Conjecture, i.e.\ dimension two, is stable. The example above demonstrates that this is not true directly as stated, but with strict convexity on the boundary it does. Thus, our result can also be viewed as an affirmative answer to a convex stability version of Min-Oo's Conjecture in dimension two. 

Let us now remark on the main ideas of the proof of Theorem \ref{main}. A more detailed description will be given shortly, after the requisite notation and setup has been properly introduced. Proceeding by way of contradiction, we obtain a sequence of convex topological disks, with c-convex boundaries whose lengths converge to the extremal value, which are bounded away from the model disk in $d_{GH}$. By the Gauss-Bonnet and Uniformization Theorems, studying this sequence amounts to studying a corresponding sequence of conformal factors for metrics on the unit disk of $\R^2$. New conformal factors for curvature $\equiv1$ disks with isometric boundaries to the original sequence are produced, to be compared to the original sequence. By fixing gauge and applying some conformal mapping and elliptic PDE theory, we obtain converging subsequences of both the comparison conformal factors and their differences with those of the original sequence. This gives subconvergence for the original sequence of conformal factors. We then upgrade this analytic convergence of conformal factors to geometric convergence of the manifolds and identify the limit as the model disk, obtaining the desired contradiction. 

\subsection{A Comment on Notation}
Throughout this paper, Riemannian metrics will often be written as being conformally equivalent to $g_{euc}$, the standard Euclidean metric on $\R^2$. We will often reference geometrical quantities defined with respect to such a metric $g=g_u=e^{2u}g_{euc}$ by the conformal factor $u$. For example, we may write $d_u$ for the distance function $d_{g_u}$ deriving from $g_u$. $\Delta$ will denote the Euclidean Laplace operator defined by $\Delta\defeq \mathrm{div}\circ\nabla$.

In general, we will denote by $\B_r(x,d)$ the open metric ball of radius $r$ about $x$ with respect to the distance $d$. In case $d=d_g$ derives from a Riemannian metric $g$, we will interchangeably use the notation $\B_r(x,g)$ and $\B_r(x,d_g)$ as is most convenient for exposition. In the special case $x=0$ and $g=g_{euc}$, we will simply write $\B_r$ for $\B_r(0,g_{euc})$. In case a measure is omitted from an integral, it is understood that the implied measure is the standard volume measure on the underlying space. 

Finally, we will follow tradition in letting $\Psi=\Psi(x)=\Psi(x|a_1,a_2,\ldots)$ denote a non-negative function, which may change from line to line, depending on a variable $x$ and any number of parameters $a_i$ with the property that if the $a_i$ are all held fixed, $\Psi\searrow 0$ as $x\to 0$. 

\section{Preliminaries}
\subsection{Our Setup and a Review of Hang-Wang's Argument}

For completeness of exposition, and to set some notation which we will use throughout the rest of the paper, we briefly recall the proof of Theorem \ref{HW} as it appears in in \cite{HW}, which we restate for convenience:

\begin{thm}[F.\ Hang and X.\ Wang \cite{HW}]
Let $(M, g)$ be a smooth, compact surface with boundary. Suppose $K_g\geqslant 1$ and $\kappa_g\geqslant c\geqslant 0$. Then $L_g(\partial M)\leqslant 2\pi/\sqrt{1+c^2}$. Moreover, equality holds if and only if $(M,g)$ is isometric to $\clo{\mathbf{B}}_{\inv\cot(c)}$. 
\end{thm}

\begin{proof}
By Gauss-Bonnet and Uniformization, there is an isometry of $(M,g)$ with $(\clo\B_1, e^{2u}g_{euc})$ for some smooth $u\on \clo\B_1\to\R$, and the curvature conditions translate to \[\begin{cases} -\Delta u=K_ge^{2u}\geqslant e^{2u} & \text{on } \B_1  \\ \partial_n u+1=\kappa_g(\gamma)e^u\geqslant ce^{u} & \text{on }\mathbb S_1.\end{cases}\] By the sub-super solution method, we can produce a \emph{constant curvature comparison factor} $v\on\clo\B_1\to\R$ to $u$, which satisfies the following: \[\begin{cases} -\Delta v=e^{2v} & \text{on }\B_1  \\ \partial_n v+1\geqslant ce^{v} & \text{on }\mathbb S_1 \\u\geqslant v & \text{on }\B_1\\ v=u & \text{on }\mathbb S_1.\end{cases}\] The \emph{constant curvature comparison disk} $(\clo\B_1, e^{2v}g_{euc})$ can therefore be realized isometrically as a smooth domain $\Omega$ in the standard $\mathbb S^2$ with boundary that is uniformly $c$-convex. Thus, the smallest geodesic disk $D\subset\mathbb S^2$ containing $\Omega$ is of radius at most $\inv\cot(c)$, and this disk has boundary length $L_{g_{rd}}(\partial D)\leqslant 2\pi/\sqrt{1+c^2}$. Since $u=v$ on $\mathbb S^1$, \[L_{g}(\gamma)=L_{e^{2v}g_{euc}}(\mathbb S^1)=L_{g_{rd}}(\partial\Omega)\leqslant 2\pi/\sqrt{1+c^2}.\] 

Now suppose equality is obtained. The construction above forces $\partial\Omega=\partial D=\partial \clo{\mathbf{B}}_{\inv\cot(c)}$, which forces the geodesic curvature of $\partial\Omega$ to be identically $c$. Thus, our comparison factor $v$ must satisfy \[\begin{cases} -\Delta v=e^{2v} & \text{on }\B_1  \\ \partial_n v+1=ce^{v} & \text{on }\mathbb S_1 \\ v=u& \text{on }\mathbb S_1.\end{cases}\] Since on $\mathbb S^1$ we have $ce^v=\partial_n v+1\geqslant \partial_n u+1\geqslant ce^u=ce^v$ we conclude that $\gamma$ also has constant geodesic curvature $c$. Setting $w=u-v$, we have that \[\begin{cases} -\Delta w\geqslant 0 & \text{on }\B_1  \\ \partial_n w=0 & \text{on }\mathbb S_1 \\w=0 & \text{on }\B_1.\end{cases}\] It then follows that $u\equiv v$ on $\clo\B_1$, which proves that $(M,g)$ is isometric to $\clo{\mathbf{B}}_{\inv\cot(c)}$. 
\end{proof}

In this paper, we are interested in the consequences of the assumption that $L_g(\gamma)$ is \emph{nearly} equal to $2\pi/\sqrt{1+c^2}$. Let then $\varepsilon>0$ be small, and consider a compact surface $(M, g)$ with boundary $\gamma$, $K_g\geqslant 1$, $\kappa_g(\gamma)\geqslant c>0$, and $L_g(\gamma)\geqslant 2\pi(1-\varepsilon)/\sqrt{1+c^2}$. We set out to prove that $(M,g)$ is Gromov-Hausdorff close to the spherical cap characterizing the equality case. Construct exactly as above the manifolds $(\clo\B_1, e^{2u}g_{euc})$ and $(\clo\B_1, e^{2v}g_{euc})$, with the latter corresponding isometrically to some domain $\Omega$ in the standard $\mathbb S^2$. The jumping off point is the following estimate for the inradius of a strictly convex domain in a sphere in terms of the length of the boundary, and is a direct adaptation of the more general Theorem 1.2 in \cite{Drach}:

\begin{thm}[Inradius Estimate for Convex Domains] 
Let $\Omega$ be a smooth convex domain in the standard $\mathbb S^2$, with boundary of length $L$ and $\kappa_g(\partial\Omega)\geqslant c>0$. Let $r_{in}$ denote the inradius of $\Omega$. Then \[r_{in}\geqslant \inv\cot(c)-\inv\cot\left(c\sec\left(\frac{L\sqrt{1+c^2}}{4}\right)\right).\]
\end{thm}

\begin{rem}
Notice the importance of the positivity of $c$ in this estimate. Indeed, the failure of an inradius lower bound for domains with boundary having segments of zero geodesic curvature allows for collapsing, and is exactly what underlies the failure of stability in the weakly convex case. 
\end{rem}

In particular $(\clo\B_1, e^{2v}g_{euc})$, when realized isometrically as a domain $\Omega$ in $\mathbb S^2$, has a large inball of radius \[r_{in}\geqslant\inv\cot(c)-\inv\cot\left(c\sec\left(\frac{\pi}{2}(1-\varepsilon)\right)\right).\] Additionally, as proven above, $\Omega$ has a geodesic disk $D$ as an outball of radius $r_{out}\leqslant \inv\cot(c)$. We thus have observed the following, which will be a fundamental lemma for us: 

\begin{lem}\label{inradius}
Let $0<\varepsilon<1$ and fix a compact surface $(M, g)$ with boundary $\gamma$, $K_g\geqslant 1$, $\kappa_g(\gamma)\geqslant c>0$, and $L_g(\gamma)\geqslant 2\pi(1-\varepsilon)/\sqrt{1+c^2}$. Let $(\clo\B_1, e^{2v}g_{euc})\leftrightarrow \Omega\subset\mathbb S^2$ be the constant curvature comparison disk in $\mathbb S^2$, which has inradius $r_{in}$ and outradius $r_{out}$. Then \[\inv\cot(c)-\Psi(\varepsilon | c)\leqslant r_{in}\leqslant r_{out}\leqslant \inv\cot(c).\]
\end{lem} 
Consequently we have volume stability, which we will use in Section \ref{GH}: \begin{prop}\label{vol conv}
With $(M, g), u, \Omega, v$ as in Lemma \ref{inradius}, we have \[\mathrm{Area}_{g_{rd}}\clo{\mathbf{B}}_{\inv\cot(c)}=2\pi\left(1-\frac{c}{\sqrt{1+c^2}}\right)\geqslant\mathrm{Area}_g(M)\geqslant2\pi\left(1-\frac{c}{\sqrt{1+c^2}}\right)-\Psi(\varepsilon|c).\]
\end{prop}
\begin{proof}
By Gauss-Bonnet, the construction of $u$ and $v$, and the hypotheses we have \[2\pi=\int_{M}K_g+\int_{\gamma}\kappa_g\geqslant \mathrm{Area}_g(M)+\frac{2\pi c(1-\varepsilon)}{\sqrt{1+c^2}}\geqslant \mathrm{Area}_{g_{rd}}(\Omega)+\frac{2\pi c(1-\varepsilon)}{\sqrt{1+c^2}}.\] Applying Lemma \ref{inradius} then yields the proposition. 
\end{proof}

\subsection{Outline of the Proof}

Here we explain the broad-strokes idea of the argument for proving our main theorem, which we restate for convenience: 
\begin{thm}[Stability of the Convex Spherical Cap]
Fix $c>0$ and let $\delta>0$. Then there exists an $\varepsilon=\varepsilon(c,\delta)>0$ such that if $(M, g)$ is a two dimensional manifold with $K_{g}\geqslant 1$, $\kappa_{g}\geqslant c>0$, and $L_{g}(\partial M)\geqslant 2\pi(1-\varepsilon)/\sqrt{1+c^2}$, then \[d_{GH}((M,g),\clo{\mathbf{B}}_{\inv\cot(c)})<\delta,\] where $\clo{\mathbf{B}}_{\inv\cot(c)}$ is a closed geodesic disk of radius $\inv\cot(c)$ in the round sphere $\mathbb S^2$. 
\end{thm}
We will prove this by way of contradiction, supposing that there exists a $\delta_0>0$ such that, for any sequence $\varepsilon_k\searrow 0$, we can find $(M_k, g_k)$ as in the statement with $L_{k}(\partial M_k)\geqslant 2\pi(1-\varepsilon_k)/\sqrt{1+c^2}$ but \[d_{GH}((M_k,g_k),\clo{\mathbf{B}}_{\inv\cot(c)})\geqslant\delta_0>0.\] 

First, let's fix the notation for the model spaces that we will be comparing our given manifolds to. Given $c\geqslant 0$, define the function $\rho_c\on\clo\B_1\to\R$ by the formula \[\rho_c(x)\defeq \log\left(\frac{2R_c}{1+|R_cx|^2}\right)\]where\[R_c\defeq \sqrt{1+c^2}-c.\] Then $(\B_1, e^{2\rho_c}g_{euc})$ is isometric, via the dilation $R_c\cdot\mathrm{Id}\on\B_1\to\B_{R_c}$, to $(\B_{R_c}, e^{2\rho_0}g_{euc})$, which under stereographic projection $\Phi\on\mathbb S^2\setminus N\to\R^2$ from the north pole $N=e_3$ is isometric to the geodesic disk of radius $\inv\cot(c)$ in $\mathbb S^2$ centered at the south pole $S=-e_3$. In other words, our model extremizer $\clo{\mathbf{B}}_{\inv\cot(c)}$ can be isometrically realized as $(\overline{\B}_1, e^{2\rho_c}g_{euc})$. Here, and throughout when convenient, we may identify $\B_r$ with $\B_r\times\{0\}\subset\R^3$, in particular when considering the stereographic projection from the standard embedding of round $\mathbb S^2$. 

Our goal is to estimate the sequence of distances $d_{GH}((\overline{\B}_1,e^{2u_k}g_{euc}),(\overline{\B}_1, e^{2\rho_c}g_{euc}))$, and extract a subsequence which converges to $0$ to derive a contradiction. At the outset, we remark that by Gromov's Compactness Theorem, any sequence of manifolds satisfying our assumptions will subconverge in the Gromov-Hausdorff (GH) sense to a compact metric space. Outright, we do not know too much about what this limit is. 

To identify the limit as being the round spherical cap, we will first control the differences of conformal factors $w_k=u_k-v_k$ and $v_k-\rho_c$. Writing for $\lambda\geqslant 1$ \[e^{\lambda u_k}-e^{\lambda\rho_c}=\left(e^{\lambda v_k}-e^{\lambda \rho_c}\right)e^{\lambda w_k}+\left(e^{\lambda w_k}-1\right)e^{\lambda\rho_c},\] we will obtain $W^{1,p}_{loc}$ convergence by showing that $e^{\lambda w_k}\to 1$ in $W^{1,p}$ and that $e^{\lambda v_k}\to e^{\lambda \rho_c}$ in $C^m_{loc}$ for any $m\geqslant 0$. The control on $w_k$ follows from standard elliptic PDE techniques and a result of H.\ Brezis-F.\ Merle in \cite{BM}. The control on $v_k-\rho_c$ is more subtle, and involves some results from the theory of conformal mappings of convex domains. 

With this control established, we can show that our sequence has a local GH sublimit on each $\clo\B_r\subset\B_1$. Using variants of the Sobolev Trace Theorem and the Bishop-Gromov Theorem, we can identify the local GH limits as spherical caps. Using the Perelman Stability Theorem, a diagonal argument with the Arzel\'a-Ascoli Theorem, and volume convergence, we then glue these local limits together to identify the model disk $(\overline{\B}_1, e^{2\rho_c}g_{euc})$ as the GH limit. Thus, we obtain a convergent subsequence to the model disk, forcing a contradiction and establishing the main theorem.

\subsection{A Remark on the Brezis-Merle and other Blow-Up Theories}

Evidently, the study of sequences of Riemannian surfaces is linked, via the uniformization process described above, to the study of sequences of solutions to the Liouville equation on a two dimensional domain: \[-\Delta u=K(x)e^{2u}\text{\quad on \quad}\Omega\subset\R^2.\] The geometric interpretation is that the metric $e^{2u}g_{euc}$ on $\Omega$ has Gaussian curvature $K$. In their seminal 1991 paper \cite{BM}, H.\ Brezis and F.\ Merle studied the blow up behavior of solutions to this equation. Their analysis, which includes a study of uniform a-priori estimates for sequences of such solutions $u$, require $L^p$ bounds on the potentials $K$ for $p>1$. We only have $L^1$ bounds on $K$, however, rendering their conclusions unavailable to us. Nonetheless, we will find great use in the following fundamental estimate from \cite{BM}:

\begin{thm}[H.\ Brezis and F.\ Merle]\label{BM Lemma}
Assume $\Omega\subset\R^2$ is a bounded domain and let $u$ be a solution of \[\begin{cases} -\Delta u=f & \text{on }\Omega  \\ u=0 &\text{on }\partial\Omega\end{cases}\] with $f\in L^1(\Omega)$. Then for every $\delta\in(0,4\pi)$, we have the estimate \[\int_{\Omega} e^{\frac{4\pi-\delta}{\|f\|_{L^1(\Omega)}}|u(x)|}\mathrm{d}x\leqslant \frac{4\pi^2}{\delta}(\mathrm{diam}\Omega)^2.\]
\end{thm}

Several authors have recently investigated possible extensions of the blow-up analysis, in particular with attention to geometric applications. For example \cite{LST}, \cite{LT}, and \cite{CL} have studied, among other things, the compactness of sequences of Riemannian surfaces with curvature bounds via an analysis of this equation. We remark that their results seem to be largely unavailable to us here, given the more general nature of our curvature bounds and the desire for identifying exact limits of converging sequences.

\section{Stability of Convex Disks}
\subsection{Controlling the Difference of Conformal Factors}\label{section w}

The goal of this section, largely a rapid-fire sequence of lemmas, is to prove the following proposition:
\begin{prop}\label{w}
Let $c>0$, $\lambda\geqslant 1$, $p\in [1,2)$, and $\varepsilon>0$ small. Given  $(\clo\B_1, e^{2u}g_{euc})$ and  $(\clo\B_1, e^{2v}g_{euc})$ as in the proof of Hang-Wang's Theorem\ref{HW} with $\kappa_u\geqslant c$ and $L_u(\mathbb S^1)\geqslant 2\pi(1-\varepsilon)/\sqrt{1+c^2}$, set $w=u-v$. Then\[\|e^{\lambda w}-1\|_{W^{1,p}(\B_1)}\leqslant\Psi(\varepsilon|c, p, \lambda).\] Moreover, the function $\Psi$ can be written down explicitly. 
\end{prop} We begin by collecting together some basic facts about $w$: 

\begin{lem} \label{Lemma 0} The difference of conformal factors $w=u-v$ satisfies the following:
\begin{itemize}
\item[(i)] $w\geqslant 0$ on $\B_1$
\item[(ii)] $-\Delta w\geqslant 0$  on $\B_1$
\item[(iii)] $w=0$ on $\mathbb S^1$
\item[(iv)] $\partial_n w\leqslant 0$ on $\mathbb S^1$
\item[(v)] $\lvert \int_{\mathbb S^1}\partial_n w\rvert\leqslant\Psi(\varepsilon|c)$.
\end{itemize}
\end{lem}

\begin{proof}
Items $(i), (ii), (iii),$ and $(iv)$ are rather immediate from the construction of $v$ via the sup-super solution method, so we focus on item $(v)$. This relies upon the inradius estimate of Lemma \ref{inradius} in a crucial way, and is in a sense the most `geometrically informed' result of the Lemma. 

We seek to estimate \[0\geqslant\int_{\mathbb S^1}\partial_n w=\int_{\mathbb S^1}(\partial_n u+1)-(\partial_n v+1)=\int_{\mathbb S^1}(\kappa_u-\kappa_v)e^u.\] By the inradius estimate \ref{inradius}, we may consider the new comparison disk $(\clo\B_1,e^{2f}g_{euc})$ where \[\begin{cases} -\Delta f=e^{2f} & \text{on }\B_1  \\ \partial_n f+1=\cot(r_{in})e^{f} & \text{on }\mathbb S_1,\end{cases}\] which is isometric to a geodesic disk in $\mathbb S^2$ of constant boundary curvature $\cot(r_{in})$, serving as an inball for $(\clo\B_1, e^{2v}g_{euc})$.  

By the Gauss-Bonnet Theorem, we have \begin{itemize}
\item[(A)] $\displaystyle{2\pi =\int_{\B_1}e^{2f}+\int_{\mathbb S^1}\cot(r_{in})e^f=\mathrm{Area}(\clo\B_1, e^{2f}g_{euc})+\int_{\mathbb S^1}\cot(r_{in})e^f}$\\
\item[(B)] $\displaystyle{2\pi =\int_{\B_1}e^{2v}+\int_{\mathbb S^1}\kappa_v e^v=\mathrm{Area}(\clo\B_1, e^{2v}g_{euc})+\int_{\mathbb S^1}\kappa_v e^v}$
\end{itemize}
and by direct comparison
\begin{itemize}
\item[(C)] $\displaystyle{\mathrm{Area}(\clo\B_1, e^{2f}g_{euc})\leqslant\mathrm{Area}(\clo\B_1, e^{2v}g_{euc}).}$
\end{itemize}
(A), (B), and (C) together imply that \[\int_{\mathbb S^1}\kappa_v e^v\leqslant \int_{\mathbb S^1}\cot(r_{in})e^f.\] Finally, observe that \[\int_{\mathbb S^1} e^f=L_f(\mathbb S^1)\leqslant \int_{\mathbb S^1} e^v= L_v(\mathbb S^1)\leqslant \frac{2\pi}{\sqrt{1+c^2}}.\] Using parts (i)-(iii) of the Lemma, the inradius estimate \ref{inradius}, together with the prior two observations, we get \begin{equation*}\begin{split}
\int_{\mathbb S^1}ce^v=\int_{\mathbb S^1}ce^u&\leqslant\int_{\mathbb S^1}\kappa_u e^u\leqslant\int_{\mathbb S^1}\kappa_v e^v\\
&\leqslant \int_{\mathbb S^1}\cot(r_{in})e^f\leqslant \int_{\mathbb S^1} (c+\Psi)e^f\leqslant\int_{\mathbb S^1} ce^v+\Psi. 
\end{split}
\end{equation*}
Seeing as though \[\int_{\mathbb S^1} ce^v=cL_v(\mathbb S^1)\in\left[\frac{2\pi c (1-\varepsilon)}{\sqrt{1+c^2}},\frac{2\pi c}{\sqrt{1+c^2}}\right]\] it follows that, as desired, \[0\geqslant\int_{\mathbb S^1}\partial_n w=\int_{\mathbb S^1}(\kappa_u-\kappa_v)e^u\geqslant -\Psi.\]
\end{proof}

\begin{lem}\label{lemma w 1}
$\|\Delta w\|_{L^1(\B_1)}\leqslant\Psi(\varepsilon |c)$. 
\end{lem}

\begin{proof}
By the divergence theorem and Lemma \ref{Lemma 0}, we have \[0\leqslant -\int_{\B_1}\Delta w=-\int_{\mathbb S^1}\partial_n w\leqslant \Psi.\]
\end{proof}

\begin{lem}\label{BM 1}
For all $\lambda\geqslant 1$, and for all $\varepsilon>0$ small enough (depending on $\lambda$), we have \[\pi\leqslant\int_{\B_1}e^{\lambda w}\leqslant 4\pi+\Psi(\varepsilon | c,\lambda).\]
\end{lem}

\begin{proof}
In Theorem \ref{BM Lemma}, take $\Omega=\B_1$, $\Psi$ as above, and $\delta=4\pi-\lambda\Psi$. For sufficiently small $\varepsilon>0$ (depending on $\lambda$), we have $\delta=\delta(\lambda)\in (0, 4\pi)$. Then (since also $w\geqslant 0)$ \[\pi\leqslant\int_{\B_1}e^{\lambda w}\leqslant \int_{\B_1} e^{\frac{\Psi}{\|\Delta w\|_{L^1(\Omega)}}|\lambda w(x)|}\leqslant \frac{4\pi^2}{4\pi-\lambda\Psi}\cdot 4=4\pi+\Psi.\] 
\end{proof}

\begin{lem}
For all $p\in [1,2)$, $\|\nabla w\|_{L^{p}(\B_1)}\leqslant\Psi(\varepsilon |c, p)$. 
\end{lem}

\begin{proof}
This is immediate from the following, concerning solutions of the Poisson equation on the unit ball: 
\begin{prop}\footnote{To prove this, one can compute directly from the Green's Representation of the solution $w$ and apply the Minkowski integral inequality.}
Fix $p\in [1,2)$. Then there exists a $C=C(p)>0$ such that if $w$ is a smooth solution of \[\begin{cases} -\Delta w=f & \text{on }\B_1  \\ w=0 & \text{on }\mathbb S_1,\end{cases}\] then\[\|\nabla w\|_{L^p(\B_1)}\leqslant C(p) \|f\|_{L^1(\B_1)}.\]
\end{prop}
\end{proof}

\begin{lem}\label{w sobolev}
For all $p\in [1,2)$, $\|w\|_{W^{1,p}(\B_1)}\leqslant\Psi(\varepsilon|c,p)$. Moreover, for all $p\in [1,\infty)$, $\|w\|_{L^{p}(\B_1)}\leqslant\Psi(\varepsilon|c,p)$. 
\end{lem}
\begin{proof}
Since $w=0$ along $\mathbb S^1$, the Poincar\'e Inequality implies the first claim in the Lemma, after which the second part follows from the Sobolev Embedding Theorem.  
\end{proof}

\begin{lem}\label{exp w sobolev}
For every $\lambda\geqslant 1$ and $p\in [1,2)$, $\|e^{\lambda w}\|_{W^{1,p}(\B_1)}^p\leqslant 4\pi+\Psi(\varepsilon|c,p,\lambda)$. In particular, $\|\nabla e^{\lambda w}\|_{L^p(\B_1)}\leqslant\Psi(\varepsilon|c,p,\lambda)$.
\end{lem}

\begin{proof}
By Lemma \ref{BM 1}, \[\|e^{\lambda w}\|_{L^p(\B_1)}^p\leqslant 4\pi+\Psi\] for any $p\in [1,\infty)$, provided $\varepsilon$ is small enough depending upon the choice of $p$. Now, observe that \[\nabla(e^{\lambda w})=\lambda e^{\lambda w}\nabla w\]so that if $p\in [1,2)$, and we choose $q\in (1,2/p)$, 
\begin{equation*}
\begin{split}
\|\nabla(e^{\lambda w})\|_{L^p(\B_1)}^p=\lambda^p\int_{\B_1} e^{\lambda p w}|\nabla w|^p&\leqslant \lambda^p\left(\int_{\B_1}e^{\lambda pq'w}\right)^\frac{1}{q'}\left(\int_{\B_1}|\nabla w|^{pq}\right)^\frac{1}{q} \\
&\leqslant \lambda^p\left(4\pi+\Psi\right)^\frac{1}{q'}\left(\int_{\B_1}|\nabla w|^{pq}\right)^\frac{1}{q}\\
&\leqslant\Psi.
\end{split}
\end{equation*} Indeed, $q\in (1,2/p)$ implies that $\lambda p q'\in (1,\infty)$ and $pq\in(1,2)$, so that Lemma \ref{BM 1} applies to the first integral factor and Lemma \ref{w sobolev} applies in the second. As such, \[\|e^{\lambda w}\|_{W^{1,p}(\B_1)}^p= \|e^{\lambda w}\|_{L^p(\B_1)}^p+\|\nabla(e^{\lambda w})\|_{L^p(\B_1)}^p\leqslant4\pi+\Psi.\]
\end{proof}

At last, we can establish the main proposition of the section:

\begin{proof}[Proof of Proposition \ref{w}]
Fix $c>0, \lambda\geqslant 1, p\in [1,2)$, $\varepsilon>0$ small, and let \[\mu:=\frac{1}{\pi}\int_{\B_1}e^{\lambda w}.\] By the Poincar\'e Inequality, \[\|e^{\lambda w}-\mu\|_{L^{p}(\B_1)}\leqslant C\|\nabla e^{\lambda w}\|_{L^{p}(\B_1)}\leqslant\Psi\] with $\Psi$ as in Lemma \ref{exp w sobolev}. Thus, by the Sobolev Trace Theorem and the fact that $w=0$ on $\mathbb S^1$, \[(2\pi)^\frac{1}{p}\vert 1-\mu\vert=\|1-\mu\|_{L^{p}(\mathbb S^1)}\leqslant C\|e^{\lambda w}-\mu\|_{W^{1,p}(\B_1)}\leqslant\Psi.\] Consequently, we find that \[\|e^{\lambda w}-1\|_{W^{1,p}(\B_1)}\leqslant\|e^{\lambda w}-\mu\|_{W^{1,p}(\B_1)}+\|\mu-1\|_{W^{1,p}(\B_1)}\leqslant\Psi.\] By the Sobolev Embedding Theorem the second claim now follows as well. 
\end{proof}

\subsection{Estimates for the Constant Curvature Comparison Metrics}\label{section v}

In this section we study the sequence of comparison disks $(\clo\B_1, e^{2v_k}g_{euc})$ as their boundary lengths go to the extremal length $2\pi/\sqrt{1+c^2}$. This is complicated by the fact that the process of representing such disks via the Uniformization Theorem has a large gauge invariance and consequent loss of compactness. 

To illustrate this, recall that we isometrically represented $(\B_1, e^{2v}g_{euc})$ by a smooth domain $\Omega$ in the round sphere, and that this choice is unique only up to action of the symmetry group $\mathcal O(3)\circlearrowright\mathbb S^2$. Under stereographic projection, such a domain is isometric to a disk of the form $(\mathcal D, e^{2\rho_0}g_{euc})$ for some smooth domain $\mathcal D\subset\R^2$, where $e^{2\rho_0}g_{euc}$ is the round metric in stereographic coordinates. We seek to relate $v$ to $\rho_c$ on $\B_1$ via pullback of $(\mathcal D, e^{2\rho_0}g_{euc})$ by an isometry $F\on\B_1\to\mathcal D$, but the ambiguity in the choice of $\Omega$, and thus $\mathcal D$ and $F$, matters a great deal insofar as estimates are concerned. 

In particular, the gauge invariance that we seek to control manifests in the following form. Given an $f\on\B_1\to\R$ and a $\phi\in\mathrm{Diff}(\B_1)$, we employ the notation $(f)_\phi\defeq f\circ \phi+\log|\phi'|$. Observe that if $\phi\in\mathrm{Conf}(\B_1)$ is a Mobius transformation, then $(\B_1,e^{2(f)_\phi}g_{euc})$ is isometric to $(\B_1, e^{2f}g_{euc})$ under the action of pullback by $\phi$. Thus, the conformal factors $f$ and $(f)_\phi$ represent the same geometric object, and therefore the corresponding curvature data in PDE form is invariant under this action of $\mathrm{Conf}(\B_1)$. For example, if $v$ is any solution of $-\Delta v=e^{2v}$ on $\B_1$ and $\phi\in\mathrm{Conf}(\B_1)$, we have (using an apostrophe to denote complex differentiation)\[-\Delta(v)_\phi=-\Delta(v\circ \phi)=-(\Delta u)\circ \phi|\phi'|^2=e^{2v\circ \phi} |\phi'|^2=e^{2(v)_\phi}.\] Here, of course, we use the fact that the Mobius transformations are holomorphic. 

To combat this gauge invariance, we will first choose the $\Omega_k$ in a way which takes uniform advantage of the quantitative inradius estimate from Lemma \ref{inradius}. The resulting isometries relating the $(\mathcal D_k, e^{2\rho_0}g_{euc})$ to the $(\B_1, e^{2v_k}g_{euc})$ then form a sequence of conformal diffeomorphisms whose images converge in a nice way. After normalizing these maps we obtain a limit mapping, and these normalizations provide us with our final choice of gauge for looking at the sequence of constant curvature comparison disks. Before proceeding, we quickly recall for the reader's convenience the necessary complex-analytic framework. 

\subsubsection{Review of Conformal Mappings}

In this section we introduce the concepts that we will need from the theory of conformal mappings in the plane. Primarily, we are concerned with the behavior of a sequence of conformal (i.\ e.\ biholomorphic) mappings $F_k\on\B_1\to\mathcal D_k$, where the domains $\mathcal D_k$ in $\R^2$ are smooth and uniformly convex. \footnote{The subject of conformal, and more broadly harmonic, mappings in the plane is wonderfully rich, and the interested reader should especially look to the books by Pommerenke \cite{Pom} and Duren \cite{Dur}.} In this section it is natural to identify $\R^2$ with $\C$, and to view a map between regions in $\R^2$ as a complex valued mapping in the standard way. 

In our applications, our domains $\mathcal D_k$ will converge to a limiting domain $\mathcal D$. The relevant form of convergence is the following notion, due to C.\ Carath\'eodory: \footnote{This definition often takes slightly different forms. The current one is convenient for us and is found in \cite{Dur}}

\begin{defn}[Kernel Convergence] 
Let $\mathcal{D}_k\subset\C$ be simply connected domains with $0\in\mathcal{D}_k$. The \emph{kernel} of the sequence $\{\mathcal{D}_k\}$ is defined to be $\{0\}$ if $0\notin\mathrm{int}(\bigcap\mathcal{D}_k)$, and otherwise is defined to be the largest domain $\mathcal D\subset\C$ containing $0$ with the property that each point of $\mathcal{D}$ possesses a neighborhood in $\mathcal{D}$ which lies in cofinitely many $\mathcal{D}_k$. We then say that $\mathcal{D}_k\to\mathcal{D}$ (with respect to $0$) in the sense of \emph{kernel convergence} if every subsequence of the collection $\{\mathcal{D}_k\}$ has the same kernel $\mathcal D$. 
\end{defn}

There are a few examples where kernel convergence is easily verified. For example, a sequence of increasing open sets containing $0$ has their union as their kernel. In our case, each of our sets will be of the form $\B_{r_k}\subset\mathcal{D}_k\subset\B_R$, where $r_k\nearrow R$, so again the kernel is simply the union of the $\mathcal D_k$. Another way to phrase the definition of the kernel, which makes the last example clear, is as follows: for each $n\geqslant 1$ let $\mathcal C_n$ denote the connected component of $\mathrm{int}\left(\mathcal{D}_n\cap \mathcal{D}_{n+1}\cap\cdots\right)$ containing $0$. If they exist for all $n$, the union of the $\mathcal C_n$ is defined to be the kernel, and otherwise it is defined to be $\{0\}$. 

The raison d'\^etre for this notion of convergence is the following famous theorem:

\begin{thm}[Carath\'eodory's Convergence Theorem\footnote{Many different conceptions of this result exist, and the one we are using here, as stated in \cite{Dur}, is convenient for our purposes. For other conceptions, the article \cite{Cara} is a great reference.}]\label{cara}
Let $\{\mathcal{D}_k\subset\C\}$ be a sequence of simply connected domains containing $0$, and $\{F_k\on\B_1\to\mathcal D_k\}$ a sequence of bijective conformal mappings with $F_k(0)=0$ and $F'_k(0)>0$. Then the $F_k$ converge to a limit function $F$ uniformly on compact subsets of $\B_1$ if and only if $\mathcal D_k\to \mathcal D\neq\C$ in the kernel sense. 

In the event of convergence with $\mathcal D=\{0\}$, $F\equiv 0$. If we have convergence with $\mathcal D\neq \{0\}$, then $\mathcal D$ is simply connected, $F$ is a bijective conformal mapping of $\B_1$ and $\mathcal D$, and the $\inv F_k$ converge locally uniformly on $\mathcal D$ to $\inv F$. 
\end{thm}

Lastly, it is worth recalling the Kellogg-Warschawski Theorem, which ensures that the limit maps we get will be smooth up to and on the boundary of our regions of interest:

\begin{thm}[Kellogg-Warschawski\footnote{As stated in \cite{Pom}.}] 
Suppose $F\on\B_1\to\mathcal D$ is a conformal bijection, where the boundary of the domain $\mathcal D$ is a Jordan curve of regularity $C^{m,\alpha}$ for some $m\geqslant 1$ and $\alpha\in(0,1)$. Then $F^{(m)}$ has an $\alpha$-H\"older continuous extension to $\clo\B_1$. In particular, if $\mathcal D$ is smooth then all derivatives of $F$ extend continuously to $\clo\B_1$.  
\end{thm}

\subsubsection{Fixing the Gauge}

The goal of this section is to prove the following proposition. Recall that we use the notation $(f)_\phi\defeq f\circ \phi+\log|\phi'|$, and that if $\phi\in\mathrm{Mob}(\B_1)$ is a Mobius transformation of $\B_1$, then $(\B_1,e^{2(f)_\phi}g_{euc})$ is isometric to $(\B_1, e^{2f}g_{euc})$ under the action of pullback by $\phi$. 

\begin{prop}
Given $c>0$, $\varepsilon_k\searrow 0$ and smooth functions $v_k\on\clo\B_1\to\R$ with \[\begin{cases} -\Delta v_k=e^{2v_k} & \text{on }\B_1  \\ \partial_n v_k+1\geqslant ce^{u_k} & \text{on }\mathbb S_1 \\\int_{\mathbb S_1} e^{v_k}\geqslant 2\pi-\varepsilon_k,\end{cases}\] we can find a sequence of Mobius transformations $\phi_k\in\mathrm{Conf}(\B_1)$ such that $(v_k)_{\phi_k}\to\rho_c$ in $C^m_{loc}(\B_1)$ for any $m\geqslant 0$.  
\end{prop}

Geometrically, of course, we are showing that the constant curvature comparison disks $(\B_1, e^{2v_k}g_{euc})$ converge in the $C^m_{loc}$ Cheeger-Gromov sense to the model space $(\B_1, e^{2\rho_c}g_{euc})$ for every $m\geqslant 0$. 

To begin proving this, we make an initial choice of gauge by conveniently choosing the domains $\Omega_k\subset\mathbb S^2$ which isometrically realize our comparison disks $(\B_1, e^{2v_k}g_{euc})$. We consider round $\mathbb S^2$ to be isometrically embedded in $\R^3$ in the standard way, and we let $\Phi\on\mathbb S^2\setminus N\to\R^2$ denote stereographic projection from the north pole $N$. Now, choose $\Omega_k$ to lie in the southern hemisphere of $\mathbb S^2$, containing the south pole $S$. Recalling that $\Omega_k$ has an outball of radius $\inv\cot(c)$ by Lemma \ref{inradius}, we can choose $\Omega_k$ so that its outball is centered at $S$. By Lemma \ref{inradius}, $\Omega_k$ also has an inball of radius $\inv\cot(c+\Psi(\varepsilon_k))$. It thus follows that with $\Omega_k$ chosen in this way, the geodesic disk of radius $2\inv\cot(c+\Psi(\varepsilon_k))-\inv\cot(c)$ centered at $S$ is contained in $\Omega_k$. As $k\to\infty$ and $\varepsilon_k\searrow 0$, this disk expands up to the fixed outball of radius $\inv\cot(c)$ for all the $\Omega_k$. 

Under $\Phi$, $\Omega_k$ is isometric to a disk of the form $(\mathcal D_k, e^{2\rho_0}g_{euc})$, where $\mathcal D_k$ is a smooth convex domain in $\R^2$ containing $0$ and contained in $\B_1$. More precisely, the $\mathcal D_k$ all lie within $\B_{R_c}$, and each contains the ball $\B_{\ell_k}$, where \[\ell_k\defeq \frac{\sin[2\inv\cot(c+\Psi(\varepsilon_k))-\inv\cot(c)]}{1+\cos[2\inv\cot(c+\Psi(\varepsilon_k))-\inv\cot(c)]}\nearrow R_c.\] Consequently, we can easily see that $\mathcal D_k\to\B_{R_c}$ in the kernel sense: we have $\B_{\ell_N}\subset\mathcal C_N\defeq\mathrm{int}(\mathcal D_N\cap \mathcal D_{N+1}\cap\cdots)\subset\B_{R_c}$, and so the kernel, which we recall as being the union of $\mathcal C_N$, is exactly $\B_{R_c}$. Denote by $F_k\on\B_1\to\mathcal D_k$ the smooth map which provides an isometry between our two disks $(\mathcal D_k, e^{2\rho_0}g_{euc})$ and $(\B_1, e^{2v_k}g_{euc})$. 

\begin{clm}
$F_k$ is a conformal diffeomorphism of $\B_1$ and $\mathcal D_k$. 
\end{clm}

\begin{proof}
Indeed, $F_k$ an isometry means that $e^{2v_k}g_{euc}=F^*_k(e^{2\rho_0}g_{euc})=e^{2\rho_0\circ F_k}F^*_kg_{euc}$. Expanding out, we find that\[(e^{2v_k}g_{euc})_{ij}=\begin{cases}
e^{2\rho_0\circ F_k}|\partial_1 F_k|^2 & i=j=1\\
e^{2\rho_0\circ F_k}\langle \partial_i F_k, \partial_j F_k\rangle & i\neq j \\
e^{2\rho_0\circ F_k}|\partial_2 F_k|^2 & i=j=2.\\
\end{cases}\] Thus, it follows that $F_k$ is a conformal diffeomorphism of $\B_1$ and $\mathcal D_k$. 
\end{proof}

One easily sees from the Cauchy-Riemann Equations that $F_k$ is a biholomorphic mapping and, if we use the complex analytic notation $(-)'$ for differentiation, $|\partial_i F_k|^2=|F'_k|^2$ for $i=1,2$. Thus,

\begin{cor}
Under the isometry $F_k$, we have $v_k=\rho_0\circ F_k+\log|F'_k|$. 
\end{cor} 

Our task now is to investigate the convergence of the maps $F_k\on\B_1\to\mathcal D_k$, and it is here where our specific choices concerning the domains $\Omega_k$ help. The main result of this section boils down to the following: 

\begin{prop}
Given $\varepsilon_k\searrow 0$, $v_k$, and $F_k$ all as above, we can find Mobius transformations $\phi_k\in\mathrm{Conf}(\B_1)$ such that the maps $\hat F_k\defeq F_k\circ\phi_k\to R_c\cdot\mathrm{Id}$ in $C^m_{loc}(\B_1)$ for any $m\geqslant 0$. In particular, $(v_k)_{\phi_k}\to\rho_c$ in $C^k_{loc}(\B_1)$ for any $k\geqslant 0$.  
\end{prop}

\begin{proof}
Choose $\eta_k\in\mathrm{Conf}(\B_1)$ such that $\tilde F_k\defeq F_k\circ\inv\eta_k$ satisfies the normalization conditions $\tilde F_k(0)=0$ and $\tilde F_k'(0)>0$. Then $\tilde F_k\on\B_1\to\mathcal D_k$ satisfies all the conditions of the Carath\'eodory Convergence Theorem \ref{cara}, so we have a conformal diffeomorphism $\tilde F\on\B_1\to\B_{R_c}$ to which the $\tilde F_k$ limit in $C^0_{loc}(\B_1)$. It follows that $\tilde F$ is simply a Mobius transformation $\psi\in\mathrm{Conf}(\B_1)$ scaled by the factor $R_c$, so $\tilde F_k\to R_c\cdot\psi$.

Now, let $\hat F_k\defeq\tilde F_k\circ\inv\psi=F_k\circ\inv\eta_k\circ\inv\psi\defeq F_k\circ\phi_k$. By the Cauchy Integral Formula, we see that $\hat F_k\to R_c\cdot\mathrm{Id}$ in $C^m_{loc}(\B_1)$ for any $m\geqslant 0$. To see that the $\phi_k$ provide a good change of gauge, we compute the pullback of $(\mathcal D_k, e^{2\rho_0}g_{euc})$ by $\hat F_k$: \begin{equation*}\begin{split}
(\B_1, e^{2(\rho_0)_{\hat F_k}}g_{euc})=\hat F^*_k (\mathcal D_k, e^{2\rho_0}g_{euc})&=\phi_k^* F^*_k(\mathcal D_k, e^{2\rho_0}g_{euc})\\
&=\phi_k^*(\B_1, e^{2v_k}g_{euc})=(\B_1, e^{2(v_k)_{\phi_k}}g_{euc}).
\end{split}
\end{equation*}
Therefore, \[(v_k)_{\phi_k}=(\rho_0)_{\hat F_k}(x)=\rho_0\circ \hat F_k(x)+\log|\hat F_k'(x)|\to \rho_0(R_c\cdot x)+\log R_c=\rho_c(x)\] in $C^m_{loc}(\B_1)$ for any $m\geqslant 0$, since $\rho_0$ is smooth on $\clo\B_1$.
\end{proof}

Now, recall that since $\phi_k\in\mathrm{Conf}(\B_1)$ the disk $(\B_1, e^{2(v_k)_{\phi_k}}g_{euc})$ is isometric to $(\B_1, e^{2v_k}g_{euc})$. Therefore, $\tilde v_k\defeq (v_k)_{\phi_k}$ satisfies all of the same curvature conditions that $v_k$ does, namely, \[\begin{cases} -\Delta \tilde v_k=e^{2\tilde v_k} & \text{on }\B_1  \\ \partial_n \tilde v_k+1\geqslant ce^{\tilde v_k} & \text{on }\mathbb S_1. \end{cases}\] Similarly, our original disk $(\B_1, e^{2u_k}g_{euc})$ is isometric to $(\B_1, e^{2(u_k)_{\phi_k}}g_{euc})$ so the functions $\tilde u_k\defeq (u_k)_{\phi_k}$ similarly satisfy all the same curvature conditions: \[\begin{cases} -\Delta \tilde u_k=K_{\tilde u_k} e^{2\tilde u_k}\geqslant e^{2\tilde u_k} & \text{on }\B_1  \\ \partial_n \tilde u_k+1\geqslant ce^{\tilde u_k} & \text{on }\mathbb S_1.\end{cases}\] Of course, the relations $\tilde u_k\geqslant \tilde v_k$ on $\B_1$, $\tilde u_k=\tilde v_k$ on $\mathbb S^1$, and \[L_{\tilde u_k}(\mathbb S^1)=\int_{\mathbb S^1}e^{\tilde u_k}=\int_{\mathbb S^1}e^{\tilde v_k}\geqslant 2\pi(1-\varepsilon_k)/\sqrt{1+c^2}\] continue to hold. Therefore, if $\tilde w_k\defeq\tilde u_k-\tilde v_k$, all of the estimates of the previous section for $w_k$ apply also to $\tilde w_k$. Our gauge fixing process thus consists of replacing $u_k$, $v_k$, and $w_k$ with their counterparts $\tilde u_k$, $\tilde v_k$, and $\tilde w_k$. These conformal factors (and their difference, respectively) all represent the same geometric objects up to isometry with the same estimates, but with the property that the new constant curvature comparison factors $\tilde v_k$ enjoy good local convergence to the model conformal factor $\rho_c$.

\emph{From here on out, for clarity of notation we will denote our gauge-corrected conformal factors and their difference with the original undecorated quantities $u_k$, $v_k$, and $w_k$, implicitly assuming a-priori correction by the maps $\phi_k$.}

\subsection{Producing a Metric Space Limit}\label{section X}

We now have the estimates that we need to show that the disks $(\B_1, e^{2u_k}g_{euc})$ converge to a limit. First, let us collect what we have shown so far. For $\lambda\geqslant 1$ we have \[e^{\lambda u_k}-e^{\lambda\rho_c}=\left(e^{\lambda v_k}-e^{\lambda \rho_c}\right)e^{\lambda w_k}+\left(e^{\lambda w_k}-1\right)e^{\lambda\rho_c}.\] By the results of Sections \ref{section w} and \ref{section v}, the pair of product terms on the right hand side tends to zero in $W^{1,p}_{loc}(\B_1)$ for every $p\in[1,2)$ and in $L^q_{loc}(\B_1)$ for every $q\in [1,\infty)$, as can be readily seen by H\"older's Inequality and the Sobolev Embedding Theorem. In particular, we get strong local Sobolev convergence of the metrics $e^{2u_k}g_{euc}$ to the model metric $e^{2\rho_c}g_{euc}$. To bootstrap this analytic convergence up to geometric convergence, we begin by showing that the distance functions of the $e^{2u_k}g_{euc}$ subconverge to a semi-definite distance function on $\B_1$, which we will then show is the distance function of the model metric. 

\begin{prop}\label{distancelim}
There exists a continuous function $d_\infty$ on $\B_1\times\B_1$ such that, upto subsequence, $d_{u_k}\to d_\infty$ in $C^0_{loc}(\B_1\times\B_1)$. 
\end{prop}

\begin{proof}
Fix $0<r<1$. For any $(x,y)\in\B_1\times\B_1$ with $x\neq y$, we have that (away from the measure zero cut loci of $y$ and $x$, respectively) \[\lvert\nabla_x d_{u_k}(x,y) \rvert=e^{2u_k(x)}\text{\quad and\quad}\lvert\nabla_y d_{u_k}(x,y) \rvert=e^{2u_k(y)}.\] Here, the norms and gradients are Euclidean. So, with $1\leqslant p<\infty$ and the $L^p(\clo\B_r)$ bound on $e^{u_k}$, we obtain \[\int_{\clo\B_r\times\clo\B_r}\lvert\nabla_x d_{u_k}(x,y) \rvert^p+\lvert\nabla_y d_{u_k}(x,y) \rvert^p\mathrm{d}x\mathrm{d}y=\int_{\clo\B_r\times\clo\B_r}e^{2pu_k(x)}+e^{2pu_k(y)}\mathrm{d}x\mathrm{d}y\leqslant C(p,r).\] Thus, $\{d_{u_k}\}$ is bounded in $W^{1,p}(\clo\B_r\times\clo\B_r)$ for any $1\leqslant p<\infty$. By the Morrey-Sobolev embedding theorem this sequence is also bounded in $C^\alpha(\clo\B_r\times\clo\B_r)$ for any $\alpha=1-4/p\in(0,1)$. By the compact embedding of H\"older spaces, we see that there is a continuous function $d_\infty$ on $\clo\B_r\times\clo\B_r$ such that, upto subsequence, $d_{u_k}\to d_\infty$ in $C^\alpha(\clo\B_r\times\clo\B_r)$ for each $\alpha\in(0,1)$. By letting $r\nearrow 1$ we thus obtain pointwise convergence of a subsequence to $d_\infty$ on $\B_1\times\B_1$ which is uniform on any compact subset. 
\end{proof}

We remark that $d_\infty$ defines a \emph{semi-metric} on $\B_1$ to which the $d_{u_k}$ limit. Symmetry and the triangle inequality follow directly from the origin of $d_\infty$ as a limit of metrics, but a priori we do not know that $d_\infty$ is positive definite. We next seek to identify $d_\infty$ with $d_{\rho_c}$, the distance function associated to the round metric on $\B_1$ with constant boundary curvature $c$.

\subsection{Identifying the Local Gromov-Hausdorff Limits}

In this section we show that, for any $r<1$, the sequence of metric spaces underlying the $\left(\overline{\mathbb B}_r, e^{2u_k}g_{euc}\right)$ converges in the Gromov-Hausdorff sense to the underlying metric space of the spherical domain $\left(\overline{\mathbb B}_r,e^{2\rho_c}g_{euc}\right)$. 

Based off of the work in the previous section our candidate $GH$ limit is $(\overline{\mathbb B}_r, d_\infty)$. However, $d_\infty$ is only known to be a semi-metric at this point, so we instead prove directly that $d_\infty=d_{\rho_c}$. The following argument identifying $d_\infty$ is inspired by one in \cite{LST}. 

We'll need the following Sobolev trace-type inequality to identify the limit. The proof follows the exact same argument as in the standard case (see, for example, \cite{EG}) where the curve $\gamma$ is the boundary of a Lipschitz domain.

\begin{prop}[A Sobolev Trace Theorem] 
Let $\Omega$ be a precompact domain in $\mathbb R^2$. Let $\gamma\subset \Omega$ be a curve of finite length which as a set is Lipschitz, in the sense that it is locally the graph of a Lipschitz function over its tangent line. 

Then, for any $p\geqslant 1$, there exists a $C=C(\Omega,\gamma, p)>0$ such that: if $u\in W^{1,p}(\Omega)$, then the trace operator $T: W^{1,p}(\Omega)\to L^p(\gamma)$ is a bounded linear operator with operator norm $\|T\|\leqslant C$. 
\end{prop}

We also need the following proposition, which is a version of the standard Bishop-Gromov Theorem\footnote{Technically, since we call for \emph{sectional} curvature bounds, this might be more accurately called a boundary version of the G\"unther volume comparison.} where we allow for the distance balls for both metrics in consideration to make contact with their respective convex boundaries: 

\begin{prop}
Let $\Omega\subset\R^2$ be a precompact, smooth set. Suppose $g$ and $g_c$ are two smooth Riemannian metrics on $\Omega$, with the properties that $\mathrm{sec}_g\geqslant\mathrm{sec}_{g_c}\equiv 1$, and $\kappa_{g}(\partial\Omega)\geqslant\kappa_{g_c}(\partial\Omega)\equiv c\geqslant 0$. Let $x\in\Omega$ and $s\in(0,\pi)$. Then \[\mathrm{vol}_{g}\left(\B_s(x,g)\subset\Omega\right)\leqslant\mathrm{vol}_{g_c}\left(\B_s(x,g_c)\subset\Omega\right)\]
\end{prop}

\begin{proof}
By convexity of the boundaries, $\Omega$ can be covered by a global normal coordinate chart in each metric, and the volume comparison follows from the standard proof. More specifically, the assumption on curvatures tells us that $g\leqslant g_c$ on $\Omega$, and expressing the volumes as integrals of $\sqrt{\det g}$ and $\sqrt{\det g_c}$ in the normal coordinates proves the claim. See for example \cite{Lee} Chapter 11.
\end{proof}

Now we are ready to prove that $d_\infty=d_{\rho_c}$. 

\begin{clm}[$d_\infty\leqslant d_{\rho_c}$]
\end{clm}  
\begin{proof}
Fix any $x,y\in\B_1$ and $r<1$ so that $x,y\in\B_r$. Let $\gamma$ be the $\rho_c$ geodesic in $\B_1$ from $x$ to $y$, which by convexity lives inside of $\B_r$. Since each $\rho_c$ geodesic in $\B_1$ is an admissible curve in the above Trace Theorem, \[e^{u_k}\to e^{\rho_c}\text{\quad in\quad} W^{1,p}(\clo \B_r) \text{\quad for all\quad} p\in [1,2)\] implies that \[e^{u_k}\to e^{\rho_c}\text{\quad in\quad} L^p(\gamma) \text{\quad for all\quad} p\in [1,2).\] Therefore, \[d_{\rho_c}(x,y)=\int_\gamma e^{\rho_c}=\lim\int_{\gamma}e^{u_k}=\lim\mathrm{L}_{u_k}(\gamma)\geqslant\liminf d_{u_k}(x,y)=d_\infty(x,y). \]
\end{proof}
\begin{clm}[$d_\infty\geqslant d_{\rho_c}$]
\end{clm}  
\begin{proof}
For the sake of contradiction, suppose that there were to exist some $x,y\in\B_r\subset\B_1$ with \[R_\infty\defeq d_{\infty}(x,y)<R\defeq d_{\rho_c}(x,y).\] Since $y\in\B_R(x,d_\infty)$ but $y\notin\B_R(x,d_{\rho_c})$, by continuity of $d_\infty$ we have that \[\mathrm{vol}_{\rho_c}\left(\B_R(x,d_{\rho_c})\right)<\mathrm{vol}_{\rho_c}\left(\B_R(x,d_\infty)\right).\] Next, choose an $\eta>0$ small and recall that $d_{u_k}\to d_\infty$ uniformly on $\B_r$. For all large enough $k$ we then have that \[\clo\B_r\cap\B_{R-\eta}(x, d_\infty)\subset\B_R(x,d_{u_k}).\] Putting these two facts together with convergence of $e^{2u_k}$ to $e^{2\rho_c}$ in $L^1(\clo\B_r)$ and applying our version of Bishop-Gromov with boundaries, we obtain 
\begin{align*}
\mathrm{vol}_{\rho_c}\left(\clo\B_r\cap\B_{R-\eta}(x,d_{\infty})\right)&=\lim \mathrm{vol}_{u_k}\left(\clo\B_r\cap\B_{R-\eta}(x,d_\infty)\right)\\
&\leqslant\liminf \mathrm{vol}_{u_k}\left(\B_R(x,d_{u_k})\right)  \\
&\leqslant \mathrm{vol}_{\rho_c}\left(\B_R(x,d_{\rho_c})\right).
\end{align*} We now send $r\nearrow 1$ and $\eta\searrow 0$ (equivalently, take the union of the sets $\clo\B_{r_j}\cap\B_{R-\eta_j}(x,d_{\infty})$ for a sequence $r_j\nearrow 1$ and $\eta_j\searrow 0$) to obtain a contradiction with our first strict volume estimate above. 
\end{proof}
Therefore, we have identified the locally uniform limit $d_\infty$ as the bonafide distance function $d_{\rho_c}$ deriving from the round metric.

\subsection{Identifying the Global Gromov-Hausdorff Limit}\label{GH}

We are now ready to show that $(\overline{\mathbb B}_1, d_{u_k})\to(\overline{\mathbb B}_1, d_{\rho_c})$ in the GH sense. By Gromov's Compactness Theorem, the $(\overline{\mathbb B}_1,  d_{u_k})$ subconverge to some metric space $(X,d_X)$. It is well known that $(X,d_X)\in\mathrm{Alex}^{\leqslant 2}(1)$ is an Alexandrov space of $\mathrm{curv}\geqslant 1$ and dimension not exceeding $2$. By volume convergence \ref{vol conv}, $(X,d_X)\in\mathrm{Alex}^2(1)$ is non-collapsed and so by the Perelman Stability Theorem, for all large enough $k\geqslant 1$, there exist homeomorphisms $\phi_k\on(\overline{\mathbb B}_1, d_{u_k})\to (X, d_X)$ with $\mathrm{dis}(\phi_k)\to 0$ (see \cite{Kap}). In particular, $(X, d_X)$ has a two-dimensional topological manifold structure.  

Without loss of generality, then, we may assume that we have a metric $d$ on $\overline{\mathbb B}_1$ such that $(\overline{\mathbb B}_1, d)\in\mathrm{Alex}^2(1)$, and that for each $k\geqslant 1$ there is a homeomorphism \[\phi_k\on(\overline{\mathbb B}_1, d_{u_k})\to (\overline {\mathbb B}_1, d)\] with diminishing distortion \[\mathrm{dis}(\phi_k)=\sup_{x,y\in\overline{\mathbb B}_1}\left\vert \phi^*_k d(x,y)-d_{u_k}(x,y)\right\vert=\sup_{x,y\in\overline{\mathbb B}_1}\left\vert d(\phi_k(x),\phi_k(y))-d_{u_k}(x,y)\right\vert\to 0.\] Using the local uniform convergence of the $d_{u_k}$ to $d_{\rho_c}$ from the previous section, we can identify $d$ as $d_{\rho_c}$ up to some choice of gauge. 

Our approach is based on a diagonal application of the Arzel\'a-Ascoli Theorem to the homeomorphisms $\phi_k$. To do this, we must first restrict to a fixed compact subdomain. Let $r<1$, and set \[\phi_k^r:=\phi_k\vert_{\overline{\mathbb B}_r}\on(\overline{\mathbb B}_r, d_{\rho_c})\to (\overline {\mathbb B}_1, d).\] Note that the maps $\phi_k^r$ are simply the set-theoretic restrictions of the $\phi_k$ to $\overline{\mathbb B}_r$, and that the distance function on $\overline{\mathbb B}_r$ is now fixed to be $d_{\rho_c}$. Of course, because $d_{u_k}$ and $d_{\rho_c}$ are both derived from smooth conformal deformations of $d_{euc}$, the $\phi_k^r$ are all homeomorphisms of $(\overline{\mathbb B}_r, d_{\rho_c})$ with their respective images in $(\overline {\mathbb B}_1, d)$. Because we have that $d_{u_k}\to d_{\rho_c}$ uniformly on $\overline{\mathbb B}_r$, passing to the fixed background metric $d_{\rho_c}$ does not damage the distortions of the $\phi_k^r$: \[\begin{split}  \mathrm{dis}(\phi_k^r)&=\sup_{x,y\in\overline{\mathbb B}_r}\left\vert (\phi^r_k)^*d(x,y)-d_{\rho_c}(x,y)\right\vert =\sup_{x,y\in\overline{\mathbb B}_r}\left\vert d(\phi_k(x),\phi_k(y))-d_{\rho_c}(x,y)\right\vert\\&\leqslant\sup_{x,y\in\overline{\mathbb B}_r}\left\vert d(\phi_k(x),\phi_k(y))-d_{u_k}(x,y)\right\vert+\sup_{x,y\in\overline{\mathbb B}_r}\left\vert d_{u_k}(x,y)-d_{\rho_c}(x,y)\right\vert \\ &\leqslant\mathrm{dis}(\phi_k)+\|d_{u_k}-d_{\rho_c}\|_{L^\infty(\overline{\mathbb B}_r\times \overline{\mathbb B}_r)}\to 0 \text{\quad as\quad} k\to\infty.\end{split}\]

We want to apply the Arzel\'a-Ascoli Theorem to the collection \[\{\phi^r_k\}_{k=1}^\infty\subset\mathcal C^0((\overline{\mathbb B}_r, d_{\rho_c}),(\overline {\mathbb B}_1, d)),\] with the latter space endowed with the topology of uniform convergence, to obtain a subsequence of the $\phi_k$ which converge uniformly on $\overline{\mathbb B}_r$ to an isometric embedding of $(\overline{\mathbb B}_r, d_{\rho_c})$ into  $(\overline {\mathbb B}_1, d)$. Doing this for a sequence $r_i\nearrow 1$ and taking a diagonal subsequence will then give us an isometric embedding of $(\mathbb B_1, d_{\rho_c})$ into $(\overline {\mathbb B}_1, d)$. 

Fix an $r_i<1$ in a sequence $r_i\nearrow 1$. To apply Arzel\'a-Ascoli, we must check that the collection $\{\phi^{r_i}_k\}_{k=1}^\infty$ is equicontinuous and pointwise pre-compact. The latter condition is easily satisfied, since the target space is a compact space. For equicontinuity, let $\varepsilon>0$ be arbitrary, and find $N\in\mathbb Z^{\geqslant 1}$ so large that for all $k\geqslant N$, \[\mathrm{dis}(\phi_k^{r_i})=\sup_{x,y\in\overline{\mathbb B}_{r_i}}\left\vert (\phi^{r_i}_k)^*d(x,y)-d_{\rho_c}(x,y)\right\vert\leqslant \varepsilon/2.\] Thus, for all $k\geqslant N$, and all $x,y\in\overline{\mathbb B}_{r_i}$, \[d(\phi^{r_i}_k(x), \phi^{r_i}_k(y))\leqslant d_{\rho_c}(x,y)+\varepsilon/2.\] Consequently, if we choose $0<\delta=\delta(\varepsilon, r_i)\leqslant\varepsilon/2$ so small that for $k=1,\ldots, N-1$ and $x,y\in\overline{\mathbb B}_{r_i}$ with $d_{\rho_c}(x,y)\leqslant\delta$ we also have \[d(\phi^{r_i}_k(x), \phi^{r_i}_k(y))\leqslant \varepsilon,\] then indeed for every $k\geqslant 1$ and $x,y\in\overline{\mathbb B}_{r_i}$ with $d_{p_c}(x,y)\leqslant\delta$ we have \[d(\phi^{r_i}_k(x), \phi^{r_i}_k(y))\leqslant \varepsilon.\] Thus the family $\{\phi^{r_i}_k\}_{k=1}^\infty$ is (uniformly) equicontinuous and pointwise pre-compact, so we may apply the Arzel\'a-Ascoli Theorem to obtain a (not-relabeled) subsequence of the $\phi_k$ so that their restrictions $\phi^{r_i}_k$ to $\overline{\mathbb B}_{r_i}$ with the $d_{\rho_c}$ distance converge uniformly to a continuous map \[\phi^{r_i}:(\overline{\mathbb B}_{r_i}, d_{\rho_c}) \to(\overline {\mathbb B}_1, d).\] We also observe that for any $k\geqslant 1$, \[\begin{split}  \mathrm{dis}(\phi^{r_i})&=\sup_{x,y\in\overline{\mathbb B}_{r_i}}\left\vert (\phi^{r_i})^*d(x,y)-d_{\rho_c}(x,y)\right\vert =\sup_{x,y\in\overline{\mathbb B}_{r_i}}\left\vert d(\phi^{r_i}(x),\phi^{r_i}(y))-d_{\rho_c}(x,y)\right\vert\\&\leqslant\sup_{x,y\in\overline{\mathbb B}_{r_i}}\left\vert d(\phi^{r_i}(x),\phi^{r_i}(y))-d(\phi^{r_i}_k(x),\phi^{r_i}_k(y))\right\vert \\ &\qquad\qquad\qquad\qquad+\sup_{x,y\in\overline{\mathbb B}_{r_i}}\left\vert d(\phi^{r_i}_k(x),\phi^{r_i}_k(y))-d_{\rho_c}(x,y)\right\vert \\ &\leqslant \|(\phi^{r_i}_k)^*d-(\phi^{r_i})^*d\|_{L^\infty(\overline{\mathbb B}_r\times \overline{\mathbb B}_r)}+\mathrm{dis}(\phi_k^{r_i})\to 0 \text{\quad as\quad} k\to\infty.\end{split}.\] In the last line, we used the uniform convergence of the $\phi^{r_i}_k$ to $\phi^{r_i}$ and the estimate 
\[\begin{split} 
\left\vert d(\phi^{r_i}_k(x), \phi^{r_i}_k(y))-d(\phi^{r_i}(x), \phi^{r_i}(y))\right\vert 
&\leqslant \left\vert d(\phi^{r_i}_k(x), \phi^{r_i}_k(y))-d(\phi^{r_i}_k(x), \phi^{r_i}(y))\right\vert \\ 
&\qquad +\left\vert d(\phi^{r_i}_k(x), \phi^{r_i}(y))-d(\phi^{r_i}(x), \phi^{r_i}(y))\right\vert \\
&\leqslant d(\phi^{r_i}_k(y), \phi^{r_i}(y))+d(\phi^{r_i}_k(x), \phi^{r_i}(x)).\end{split}\] Thus, $\phi^{r_i}$ is a continuous map of vanishing distortion, and therefore gives an isometric embedding of $(\overline{\mathbb B}_{r_i}, d_{\rho_c})$ into  $(\overline {\mathbb B}_1, d)$. 

We apply this process in subsequence to each $r_i$, and take a diagonal subsequence. This yields a sequence of domains $\overline{\mathbb B}_{r_i}\nearrow \mathbb B_1$ and isometric embeddings $\phi^{r_i}:(\overline{\mathbb B}_{r_i}, d_{\rho_c})\to(\overline{\mathbb B}_1, d)$ with $\phi^{r_{i+1}}=\phi^{r_i}$ on $\overline{\mathbb B}_{r_i}$. The canonically induced map $\phi\on(\mathbb B_1, d_{\rho_c})\to (\overline{\mathbb B}_1, d)$ is thus a homeomorphism of vanishing distortion onto its image. That is, we have an isometric embedding of $(\mathbb B_1, d_{\rho_c})$ into $(\overline{\mathbb B}_1, d)$. By continuity, and $\mathrm{dis}(\phi)=0$, $\phi$ extends uniquely to a well defined isometric embedding \[\phi\on(\overline{\mathbb B}_1, d_{\rho_c})\to (\overline{\mathbb B}_1, d).\] 

All that remains to be shown is that $\phi$ is onto. By volume stability \ref{vol conv} and the fact that the Hausdorff measures of $(\overline{\mathbb B}_1, d_{u_k})$ weakly converge to the Hausdorff measure of $(\overline{\mathbb B}_1, d)$ (see \cite{BBI} Theorem 10.10.10), the (open) complement of $\phi(\overline{\mathbb B}_1)$ in $(\overline{\mathbb B}_1, d)$ must have measure zero, and thus must be empty. Indeed, by the topological regularity of $(\overline{\mathbb B}_1, d)$ any open subset must contain a $d$-metric ball completely in the interior of $(\overline{\mathbb B}_1, d)$, which itself lies in the domain of a bi-Lipschitz coordinate chart mapping to $\mathbb R^2$ (see \cite{BBI} Theorem 10.8.18 and the following remark). Therefore, the $\mathcal H^2_d$ measure of any non-empty open subset of $(\overline{\mathbb B}_1, d)$ is non-zero. We therefore conclude that $\phi$ is an isometry between the GH limit $(\overline{\mathbb B}_1, d)$ of our original sequence $\{(\overline{\mathbb B}_1, d_{u_k})\}$ and the model disk $(\overline{\mathbb B}_1, d_{\rho_c})$.

\subsection{Proof of the Main Theorem}

To summarize, we supposed that there were a $\delta_0>0$ such that, for every sequence $\varepsilon_k\searrow 0$, we could find an $(M_k, g_k)$ as in the statement with $L_{k}(\partial M_k)\geqslant 2\pi(1-\varepsilon_k)/\sqrt{1+c^2}$ but \[d_{GH}((M_k,g_k),\clo{\mathbf{B}}_{\inv\cot(c)})\geqslant\delta_0>0.\] As shown in the previous sections, any such sequence of disks will subconverge to $\clo{\mathbf{B}}_{\inv\cot(c)}$, contradicting the existence of $\delta_0$. Hence, the main result follows.

\section{Acknowledgments}

The author would like to thank Renato Bettiol and their adviser Davi M\'aximo for being calming voices when an error in an early draft was found, as well as the Fields Institute for its hospitality during the Thematic Program on Nonsmooth Riemannian and Lorentzian Geometry. Much of the final work was completed there during the author's visit. 

\section{Data Availability Statement}
Data sharing not applicable to this article as no datasets were generated or analysed during the current study.


\bibliographystyle{alpha}
\bibliography{Stab_Conv_Disks_v5_arxiv.bib}
\nopagebreak
\end{document}